\documentclass[a4paper,11pt]{article}
\usepackage[margin=1in]{geometry}
\usepackage{amsmath,amsthm,amssymb}
\usepackage{tikz,enumitem,enumerate}
\usepackage{hyperref}
\hypersetup{
  colorlinks=true,
  citecolor=blue,
  linkcolor=black,
  filecolor=magenta,
  urlcolor=cyan
}

\newtheorem{theorem}{Theorem}
\newtheorem{question}{Question}
\newtheorem{corollary}{Corollary}

\newtheorem{lemma}{Lemma}
\newtheorem{remark}{Remark}

\newtheorem{observation}{Observation}
\newtheorem{operation}{Operation}
\newtheoremstyle{case}{}{}{}{}{}{:}{ }{}
\theoremstyle{case}

\title{Kempe Changes in $H$-Free Graphs}

\author{
  Manoj Belavadi\thanks{%
    Department of Mathematics, Wilfrid Laurier University, Waterloo, ON, Canada, N2L 3C5. Supported by NSERC grant RGPIN-2025-05846. Emails: 
    \href{mailto:mbelavadi@wlu.ca}{mbelavadi@wlu.ca}, 
    \href{mailto:kcameron@wlu.ca}{kcameron@wlu.ca}.}
  \and 
  Kathie Cameron\footnotemark[\value{footnote}]
}

\date{\today}

\begin{document}
\maketitle

\begin{abstract}

Given a $k$-colouring of a graph $G$ and two of the colours, a {\em Kempe chain} is a connected component of the subgraph of $G$ induced by the vertices coloured with one of these two colours. A {\em Kempe swap} changes one colouring into another by interchanging the colours of the vertices in a Kempe chain. Two colourings are {\em Kempe equivalent} if each can be obtained from the other by a series of Kempe swaps; the set of Kempe equivalent colourings is called a {\em Kempe class}. For a graph $G$, let $\chi(G)$ denote its chromatic number and let  $\mathcal{C}_{k}(G)$ denote the set of all $k$-colourings of $G$. We say $G$ is \emph{Kempe connected} if for all $k\ge \chi(G)$, $\mathcal{C}_{k}(G)$ forms 
a Kempe class. For a graph $H$, graph $G$ is called $H${\em-free} if no induced subgraph of $G$ is isomorphic to $H$. 

We prove that every $H$-free graph is Kempe connected if and only if $H$ is an induced subgraph of the path on four vertices, $P_4$. 

The graph 2$K_2$ consists of four vertices and two edges which are not adjacent. We prove that for all %$k \ge 4$ and $p\ge 0$ 
$p \ge 0$, 
there is a $k$-colourable 2$K_2$-free graph $G$ such that $\mathcal{C}_{k+p}(G)$ does not form a Kempe class.

%We give an operation that transforms a $k$-colourable graph with a Kempe frozen ($k+1$)-colouring into a  
%($k+1$)-colourable graph with a Kempe frozen ($k+2$)-colouring.  

\end{abstract}

\section{Introduction}

In this paper, all graphs are finite and simple. For a positive integer $k$, a $k$-colouring of $G$ is a mapping from the vertex-set $V(G)$ of $G$ to a set of colours $\{1,2,\dots,k\}$ such that no pair of adjacent vertices receive the same colour. A graph $G$ is called $k$-colourable if it admits a $k$-colouring, and the \emph{chromatic number} of $G$, denoted $\chi(G)$, is the minimum number of colours required to colour $G$. Let $\mathcal{C}_{k}(G)$ denote the set of all $k$-colourings of $G$. A set of vertices of a graph $G$ is called an {\em independent set} if no two of them are adjacent. A colouring of a graph partitions the vertices into independent sets called {\em colour classes}. The subgraph of a graph $G=(V,E)$ {\em induced} by $U \subseteq V$ is the graph with vertex-set $U$ and edge-set all edges of $G$ with both ends in $U$. 

Given a $k$-colouring of a graph $G$ and two of the colours, a {\em Kempe chain} is a connected component of the subgraph of $G$ induced by the union of the two corresponding colour classes. A {\em Kempe swap} changes one colouring of $G$ into another by interchanging the colours of the vertices in a Kempe chain. Two colourings are {\em Kempe equivalent} if each can be obtained from the other by a series of Kempe swaps; the set of Kempe equivalent colourings is called a {\em Kempe class}. %If $\mathcal{C}_{k}(G)$ forms %a Kempe class, then we say that the graph is $k$-{\em Kempe~mixing}. 
We say a graph $G$ is   \emph{Kempe connected} if for all $k\ge \chi(G)$, $\mathcal{C}_{k}(G)$ 
forms a Kempe class.
%is $k$-Kempe mixing. 

Kempe swaps were introduced by Kempe in his incorrect proof of the Four Colour Theorem \cite{Kempe}. They are used in the proof of the Five Colour Theorem and were important in the first and second proofs of the Four Colour Theorem \cite{4CThm, 4CThmRSST}. They were used by Meyniel and Las Vergnas to prove that all 5-colourings of a planar graph \cite{Meyniel} and, more generally,  of a $K_5$-minor-free graph \cite{LVMeyniel} are Kempe equivalent. They are also useful in the study of edge-colouring (see for example,  \cite{ECSurvey, KempeForEC, KiersteadEC, Narboni, Vizing64, Vizing65}). For more background and problems on Kempe swaps, see a survey by Mohar \cite{Mohar} and a recent paper by Cranston and Feghali \cite{CranstonFeghali}.

For a graph $H$, graph $G$ is called $H$-{\em free} if no induced subgraph of $G$ is isomorphic to $H$. For a collection $\mathcal{H}$ of graphs, graph $G$ is $\mathcal{H}$-{\em free} if $G$ is $H$-free for every $H \in \mathcal{H}$; $\mathcal{H}$-free classes of graphs are called {\em hereditary}. Let $P_n$, $C_n$, and $K_n$ denote the path, cycle, and complete graph on $n$ vertices, respectively. The complete graph $K_3$ is also known as a {\em triangle}. The {\em complement} $\overline{G}$ of a graph $G$  is obtained from $G$ by interchanging edges and non-edges. For two vertex-disjoint graphs $G$ and $H$, the {\em disjoint union} of $G$ and $H$, denoted by $G + H$, is the graph with vertex-set $V(G) \cup V(H)$ and edge-set $E(G) \cup E(H)$. For a positive integer $r$, 
the disjoint union of $r$ copies of $G$ is denoted by $rG$. For a vertex $v$ of graph $G$, the
open neighbourhood of $v$, denoted $N(v)$, is the set of vertices adjacent to $v$ in $G$. The closed neighbourhood $N[v]$ of $v$ is 
$N(v) \cup \{v\}$.

Hereditary classes $\mathcal{G}$ of graphs where every graph in the class is  known to be Kempe connected include chordal graphs \cite{MoritzChordal}, bipartite graphs \cite{Mohar}, and $P_4$-free graphs (which are also called   cographs)  \cite{BonamyCographsKempeConn}. The last class is most important for this paper:

\begin{theorem}[\cite{BonamyCographsKempeConn}] \label{thm:p4-free}
    Every $P_4$-free graph is Kempe connected.
\end{theorem}

In this paper, we prove: 

\begin{theorem}
\label{thm:p4-free only}
Every $H$-free graph is Kempe connected if and only if $H$ is an induced subgraph of $P_4$. 
\end{theorem}

We also prove:
\begin{theorem}
\label{thm:2K2-free not KC}
For all 
%$k \ge 4$ and 
$p\ge 0$, there is a $k$-colourable 2$K_2$-free graph $G$ such that $\mathcal{C}_{k+p}(G)$ does not form a Kempe class.
\end{theorem}

We have previously studied what can be considered a specialization of Kempe swaps: In reconfiguration of graph colourings, we seek to transform one $k$-colouring into another by changing the colour of a single vertex at a time, always maintaining a $k$-colouring (thus always changing the colour of a vertex to a colour which does not appear in its neighbourhood). Note that such a single vertex recolouring corresponds to a Kempe swap on a Kempe chain consisting of just one vertex; this is sometimes called a {\em trivial Kempe swap}. 
For more on reconfiguration, see the surveys by van den Heuvel \cite{Heuvel2013} and Nishimura \cite{Nishimura}.

A graph $G$ is $k${\em -mixing} if any $k$-colouring can be transformed into any other by single-vertex recolouring steps described above. In a $k$-colouring of $K_k$, it is not possible to change the colour of any vertex, so it is common to study reconfiguration of colourings with at least one colour more than the chromatic number. We call a graph $G$ {\em recolourable} if it is $k$-mixing for every $k > \chi(G)$.

In \cite{BCM}, with Owen Merkel, we proved:

\begin{theorem}[\cite{BCM}]
\label{thm:BCM}
Every $H$-free graph $G$ is recolourable if and only if $H$ is an induced subgraph of $P_4$ or of $P_3+P_1$.
\end{theorem}

Our main result, Theorem \ref{thm:p4-free only}, on Kempe connectedness is a parallel result to Theorem \ref{thm:BCM} on reconfiguration of colourings.

%We mention one more difference in the study of Kempe swaps and single-vertex recolouring. In recolouring, two colourings are considered to be different if they differ in colour on a single vertex. When studying Kempe swaps, two colourings are considered to be the same if they give the same partition of the vertices. 

\section{Proof of Theorem \ref{thm:p4-free only}}
%{$P_4$-free graphs}  

A $k$-colouring of a graph $G$ is called {\em frozen} if for every vertex $v \in V(G)$, $v$ is adjacent to a vertex of every colour different from its colour; in other words, a $k$-colouring is frozen if for each vertex $v$ and for each of the $k$ colours $c$, the closed neighbourhood $N[v]$ of $v$ contains a vertex of colour $c$. 
A $k$-colouring of a graph $G$ is called {\em Kempe frozen}
%, abbreviated {\em K-frozen}, 
if each of the $k$ colours appears on some vertex and for any two colours, the subgraph induced by the union of the corresponding colour classes is connected. 

\begin{observation}
If a $k$-colouring of a graph is Kempe frozen, then it is frozen.
\end{observation}

\begin{lemma}\label{lem:Kempefrozen}
Let $G$ be a graph which has at least two $k$-colourings which give different partitions of the vertices. If $G$ has a $k$-colouring which is  Kempe frozen, then the $k$-colourings of $G$ do not form a Kempe class.
\end{lemma}

\begin{proof}
Assume the hypotheses. Let $\gamma$ be a $k$-colouring which is Kempe frozen. Then, starting with $\gamma$, any Kempe swap only interchanges the colours in two of the colour classes, so we cannot obtain a colouring whose partition into colour classes is different from that of the starting colouring.
\end{proof}

\begin{lemma}\label{lem:triangle-free}
    There exists a triangle-free graph which is not Kempe connected.
\end{lemma}

\begin{proof}
In \cite{ReedConj}, a class of triangle-free graphs each of which is Kempe frozen is given. The smallest graph in that class has 15 vertices. We give a graph on 14 vertices: 
    see Figure \ref{fig:triangle-free} for a triangle-free graph $G$ with $\chi(G) = 3$ such that $\mathcal{C}_3(G)$ does not form a Kempe class.

In the colouring on the left, for any two of the colours, the subgraph induced by the union of the two colour classes is connected; i.e., the colouring is Kempe frozen. The two colourings give different partitions of the vertices into colour classes. Thus by Lemma \ref{lem:Kempefrozen}, the set of all 3-colourings does not form a Kempe class. 
\end{proof}
 
A {\em Hamiltonian path} in a graph $G$ is a path which contains all the vertices of $G$ and a  \emph{Hamiltonian cycle} in $G$ is a cycle which contains all the vertices of $G$.

\begin{lemma}\label{lem:c4-free}
There exists a $C_4$-free graph which is not Kempe connected.
\end{lemma}
\begin{proof}
    See Figure \ref{fig:C4-free} for a $C_4$-free graph $G$ with $\chi(G) = 3$ such that $\mathcal{C}_3(G)$ does not form a Kempe class.

In the colouring on the left, for any two of the colours, the subgraph induced by the union of the two colour classes is a Hamiltonian path and thus connected; i.e., the colouring is Kempe frozen. The two colourings give different partitions of the vertices into colour classes. Thus by Lemma \ref{lem:Kempefrozen}, the set of 3-colourings does  not form a Kempe class. 

To see that the graph does not contain a $C_4$, first note that for the colouring on the left, any two colour classes induce a Hamiltonian path, so there can't be a $C_4$ that uses vertices of only two colours. So any $C_4$ would have two vertices, say $a$ and $b$, of one colour and one of each of the other two colours, say vertices $c$ and $d$; this would mean that $c$ and $d$ have the same neighbours in $a$'s colour class, but that never happens. 
\end{proof}

\begin{corollary} \label{lem:C4free}
For all $k \ge 3$ there exists a $k$-colourable $C_4$-free graph $H_k$ such that $\mathcal{C}_k(H_k)$ does not form a Kempe class.
\end{corollary}

\begin{proof}
Let $G$ be the graph of Figure \ref{fig:C4-free}. By Lemma \ref{lem:C4free}, $\chi(G) = 3$ and $\mathcal{C}_3(G)$ does not form a Kempe class. For $k \ge 4$, by taking the join of $G$ and $K_{k-3}$, we obtain a $k$-colourable graph $H_k$ for which $\mathcal{C}_k(H_k)$ does not form a Kempe class. 
\end{proof}

\begin{figure}[h] %triangle-free graph
    \centering
    \begin{tikzpicture}[scale=0.6]
    \tikzstyle{vertex}=[circle, draw, fill=black, inner sep=0pt, minimum size=5pt]  
    
    \node[vertex, label= left: 2] (0) at (-4.800000000000001, 4.857142857142858) {};
    \node[vertex, label= right: 1] (1) at (1.8999998910086493, 4.842857142857143) {};
    \node[vertex, label= left: 1] (2) at (-4.857142857142858, -2.914285932268416) {};
    \node[vertex, label= right:3] (3) at (1.999999891008649, -2.985714503696987) {};
    \node[vertex, label= above: 3] (4) at (-3.6428572518484934, 3.9000000000000004) {};
    \node[vertex, label= left: 2] (5) at (-2.6571429661342076, 3.0428571428571427) {};
    \node[vertex, label= above: 3] (6) at (-1.514285823277064, 2.0714287894112715) {};
    \node[vertex, label= above: 1] (7) at (-0.342857251848494, 3.0142857142857142) {};
    \node[vertex, label= above: 3] (8) at (0.6714284624372215, 3.8142857142857145) {};
    \node[vertex, label= above: 2] (9) at (-3.7142858232770646, -2.028571646554129) {};
    \node[vertex, label= above: 1] (10) at (-2.614285823277065, -1.1428569248744422) {};
    \node[vertex, label= right: 2] (11) at (-1.4571429661342081, -0.2571426391601559) {};
    \node[vertex, label= right: 3] (12) at (-0.4857143947056368, -1.0142854963030137) {};
    \node[vertex, label= right: 2] (13) at (0.7285713195800785, -1.985714503696987) {};

    \draw (2)--(0); \draw (0)--(1); \draw (1)--(3); \draw (3)--(2); \draw (0)--(4); \draw (4)--(5); \draw (5)--(6); \draw (6)--(7); \draw (7)--(8); \draw (8)--(1); \draw (2)--(9); \draw (9)--(10); \draw (10)--(11); \draw (11)--(12); \draw (12)--(13); \draw (13)--(3); \draw (6)--(11); \draw (0)--(10); \draw (5)--(2); \draw (12)--(9); \draw (10)--(13); \draw (7)--(4);
    \draw (5) .. controls +(up:1cm) and +(up:0.5cm) .. (8); \draw (10)--(8); \draw (1)--(12); \draw (7)--(13); 
    \end{tikzpicture}
    \hspace{20mm}
    \begin{tikzpicture}[scale=0.6]
    \tikzstyle{vertex}=[circle, draw, fill=black, inner sep=0pt, minimum size=5pt]  
    
    \node[vertex, label= left: 2] (0) at (-4.800000000000001, 4.857142857142858) {};
    \node[vertex, label= right: 1] (1) at (1.8999998910086493, 4.842857142857143) {};
    \node[vertex, label= left: 1] (2) at (-4.857142857142858, -2.914285932268416) {};
    \node[vertex, label= right:3] (3) at (1.999999891008649, -2.985714503696987) {};
    \node[vertex, label= above: 1] (4) at (-3.6428572518484934, 3.9000000000000004) {};
    \node[vertex, label= left: 3] (5) at (-2.6571429661342076, 3.0428571428571427) {};
    \node[vertex, label= above: 1] (6) at (-1.514285823277064, 2.0714287894112715) {};
    \node[vertex, label= above: 3] (7) at (-0.342857251848494, 3.0142857142857142) {};
    \node[vertex, label= above: 2] (8) at (0.6714284624372215, 3.8142857142857145) {};
    \node[vertex, label= above: 2] (9) at (-3.7142858232770646, -2.028571646554129) {};
    \node[vertex, label= above: 1] (10) at (-2.614285823277065, -1.1428569248744422) {};
    \node[vertex, label= right: 2] (11) at (-1.4571429661342081, -0.2571426391601559) {};
    \node[vertex, label= right: 3] (12) at (-0.4857143947056368, -1.0142854963030137) {};
    \node[vertex, label= right: 2] (13) at (0.7285713195800785, -1.985714503696987) {};

    \draw (2)--(0); \draw (0)--(1); \draw (1)--(3); \draw (3)--(2); \draw (0)--(4); \draw (4)--(5); \draw (5)--(6); \draw (6)--(7); \draw (7)--(8); \draw (8)--(1); \draw (2)--(9); \draw (9)--(10); \draw (10)--(11); \draw (11)--(12); \draw (12)--(13); \draw (13)--(3); \draw (6)--(11); \draw (0)--(10); \draw (5)--(2); \draw (12)--(9); \draw (10)--(13); \draw (7)--(4);
    \draw (5) .. controls +(up:1cm) and +(up:0.5cm) .. (8); \draw (10)--(8); \draw (1)--(12); \draw (7)--(13); 
    \end{tikzpicture}    
    
    \caption{Two 3-colourings of a triangle-free graph which are not Kempe equivalent}
    \label{fig:triangle-free}
\end{figure}
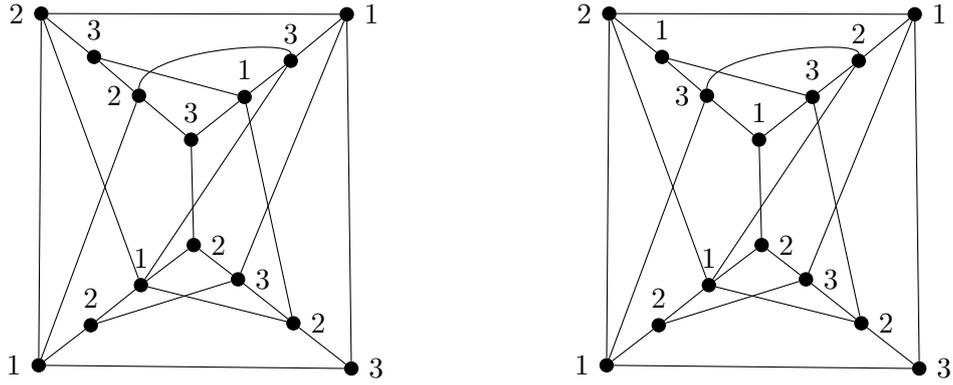

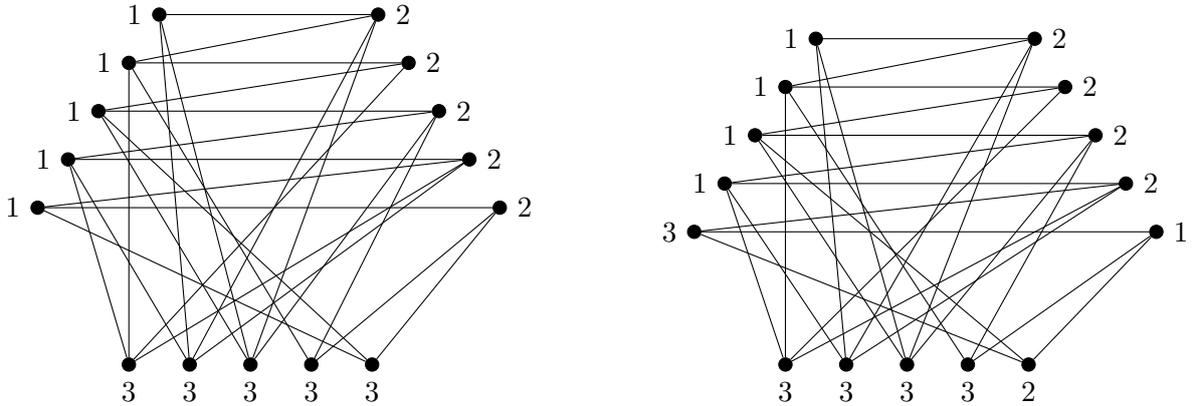
\begin{figure} %C_4-free graph
    \centering
    \begin{tikzpicture}[scale=0.8]
    \tikzstyle{vertex}=[circle, draw, fill=black, inner sep=0pt, minimum size=5pt]  
    \node[vertex, label= left: 1] (0) at (-4.2, 1.6) {};
    \node[vertex, label= left: 1] (1) at (-3.7, 2.4) {};
    \node[vertex, label= left: 1] (2) at (-3.2, 3.2) {};
    \node[vertex, label= left: 1] (3) at (-2.7, 4.0) {};
    \node[vertex, label= left: 1] (4) at (-2.2, 4.8) {};
    \node[vertex, label= right: 2] (5) at (1.4, 4.8) {};
    \node[vertex, label= right: 2] (6) at (1.9, 4.0) {};
    \node[vertex, label= right: 2] (7) at (2.4, 3.2) {};
    \node[vertex, label= right: 2] (8) at (2.9, 2.4) {};
    \node[vertex, label= right: 2] (9) at (3.4, 1.6) {};
    \node[vertex, label= below: 3] (10) at (-2.7, -1) {};
    \node[vertex, label= below: 3] (11) at (-1.7, -1) {};
    \node[vertex, label= below: 3] (12) at (-0.7, -1) {};
    \node[vertex, label= below: 3] (13) at (0.3, -1) {};
    \node[vertex, label= below: 3] (14) at (1.3, -1) {};

    \draw (4)--(5); \draw (5)--(3); \draw (3)--(6);\draw (6)--(2); \draw (2)--(7); \draw (7)--(1); \draw (1)--(8);\draw (8)--(0); \draw (0)--(9); \draw (9)--(14); \draw (13)--(9); \draw (13)--(7); \draw (7)--(12); \draw (12)--(5); \draw (5)--(11);\draw (11)--(8); \draw (8)--(10);\draw (10)--(6); \draw (0)--(14); \draw (14)--(2); \draw (2)--(12); \draw (12)--(4); \draw (4)--(11); \draw (11)--(1); \draw (1)--(10); \draw (10)--(3); \draw (3)--(13);
    \end{tikzpicture}
    \hspace{12mm}
    \begin{tikzpicture}[scale=0.8]
    \tikzstyle{vertex}=[circle, draw, fill=black, inner sep=0pt, minimum size=5pt]  
    \node[vertex, label= left: 3] (0) at (-4.2, 1.6) {};
    \node[vertex, label= left: 1] (1) at (-3.7, 2.4) {};
    \node[vertex, label= left: 1] (2) at (-3.2, 3.2) {};
    \node[vertex, label= left: 1] (3) at (-2.7, 4) {};
    \node[vertex, label= left: 1] (4) at (-2.2, 4.8) {};
    \node[vertex, label= right: 2] (5) at (1.4, 4.8) {};
    \node[vertex, label= right: 2] (6) at (1.9, 4.0) {};
    \node[vertex, label= right: 2] (7) at (2.4, 3.2) {};
    \node[vertex, label= right: 2] (8) at (2.9, 2.4) {};
    \node[vertex, label= right: 1] (9) at (3.4, 1.6) {};
    \node[vertex, label= below: 3] (10) at (-2.7, -0.6) {};
    \node[vertex, label= below: 3] (11) at (-1.7, -0.6) {};
    \node[vertex, label= below: 3] (12) at (-0.7, -0.6) {};
    \node[vertex, label= below: 3] (13) at (0.3, -0.6) {};
    \node[vertex, label= below: 2] (14) at (1.3, -0.6) {};

    \draw (4)--(5); \draw (5)--(3); \draw (3)--(6);\draw (6)--(2); \draw (2)--(7); \draw (7)--(1); \draw (1)--(8);\draw (8)--(0); \draw (0)--(9); \draw (9)--(14); \draw (13)--(9); \draw (13)--(7); \draw (7)--(12); \draw (12)--(5); \draw (5)--(11);\draw (11)--(8); \draw (8)--(10);\draw (10)--(6); \draw (0)--(14); \draw (14)--(2); \draw (2)--(12); \draw (12)--(4); \draw (4)--(11); \draw (11)--(1); \draw (1)--(10); \draw (10)--(3); \draw (3)--(13);
    \end{tikzpicture}
    \caption{Two 3-colourings of a $C_4$-free graph which are not Kempe equivalent}
    \label{fig:C4-free}
\end{figure}

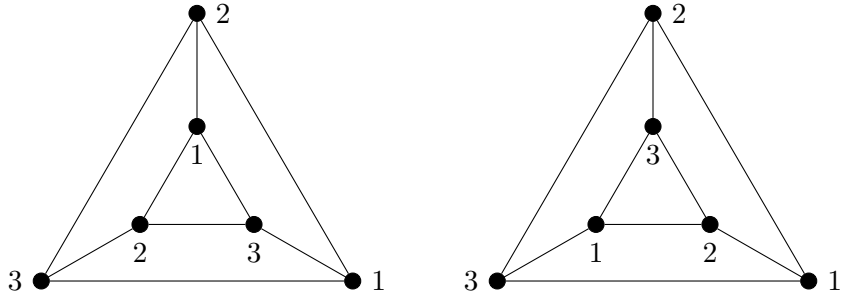
\begin{figure}[!ht] %triangular prism
\center
\begin{tikzpicture}
    \tikzstyle{whitenode}=[draw,circle,fill=white,minimum size=9pt,inner sep=0pt]
    \tikzstyle{blacknode}=[draw,circle,fill=black,minimum size=6pt,inner sep=0pt]
\draw (0,0) node[blacknode] (c) [label=-90:$2$] {} 
 ++(0:1.5cm) node[blacknode] (b) [label=-90:$3$] {}
 ++(120:1.5cm) node[blacknode] (a) [label=-90:$1$] {};
 
 \draw (a)
-- ++(90:1.5cm) node[blacknode] (a2) [label=0:$2$] {};
  \draw (b)
-- ++(-30:1.5cm) node[blacknode] (b2) [label=0:$1$] {};
  \draw (c)
-- ++(-150:1.5cm) node[blacknode] (c2) [label=180:$3$] {};
 
 \draw (a) -- (b);
 \draw (c) -- (b);
 \draw (a) -- (c);
 \draw (a2) -- (b2);
 \draw (c2) -- (b2);
 \draw (a2) -- (c2);
 
 \draw (6,0) node[blacknode] (c) [label=-90:$1$] {} 
 ++(0:1.5cm) node[blacknode] (b) [label=-90:$2$] {}
 ++(120:1.5cm) node[blacknode] (a) [label=-90:$3$] {};
 
 \draw (a)
-- ++(90:1.5cm) node[blacknode] (a2) [label=0:$2$] {};
  \draw (b)
-- ++(-30:1.5cm) node[blacknode] (b2) [label=0:$1$] {};
  \draw (c)
-- ++(-150:1.5cm) node[blacknode] (c2) [label=180:$3$] {};
 
 \draw (a) -- (b);
 \draw (c) -- (b);
 \draw (a) -- (c);
 \draw (a2) -- (b2);
 \draw (c2) -- (b2);
 \draw (a2) -- (c2);

\end{tikzpicture}
\caption{Two $3$-colourings of the triangular prism which are not Kempe equivalent \cite{Heuvel2013}.}
\label{fig:prism}
\end{figure}

\begin{lemma}\label{lem:4-vertex}
    Suppose $H$ is a graph on at most four vertices. Every $H$-free graph is Kempe connected if and only if $H$ is an induced subgraph of $P_4$.
\end{lemma}
\begin{proof}
    By Theorem \ref{thm:p4-free}, if $H$ is an induced subgraph of $P_4$ then every $H$-free graph is Kempe connected. So we may assume that $H$ is not an induced subgraph of $P_4$.
    
    Suppose $H$ contains a triangle, then see Figure \ref{fig:triangle-free} for an $H$-free graph which is not Kempe connected. Suppose $H$ contains a 3$K_1$, then see Figure \ref{fig:prism} for an $H$-free graph which is not Kempe connected. There are two (triangle, 3$K_1$)-free graphs on at most four vertices which are not induced subgraphs of $P_4$. They are 2$K_2$ and $C_4$. See Figure \ref{fig:prism} for a 2$K_2$-free graph which is not Kempe connected, and see Figure \ref{fig:C4-free} for a $C_4$-free graph which is not Kempe connected.
\end{proof}

\begin{lemma}\label{lem:5-vertex}
    For every graph $H$ on five vertices, there exists an $H$-free graph which is not Kempe connected.
\end{lemma}
\begin{proof}
    Note that there are 34 graphs on 5 vertices. We partition them into three groups: Group 1 consists of the 20 graphs which contain a triangle. Group 2 consists of the 13 graphs which are triangle-free but contain a 3$K_1$. Group 3 consists of $C_5$. These facts can be easily verified (for example, on houseofgraphs.org).

    If $H$ is in Group 1, then the lemma follows from Lemma \ref{lem:triangle-free}. The triangular prism $\overline{C_6}$ in Figure \ref{fig:prism} is a (3$K_1$, $C_5$)-free graph which is not Kempe connected. Thus, if $H$ is in Group 2 or Group 3, then there is an $H$-free graph which is not Kempe connected.
\end{proof}

Theorem \ref{thm:p4-free only} follows from Lemmas \ref{lem:4-vertex} and \ref{lem:5-vertex}.

\section{Kempe Frozen Colourings of \texorpdfstring{$2K_2$} ~-free graphs \label{sec:2K2-free K-frozen}}

In \cite{frozen}, several classes of $2K_2$-free graphs with frozen colourings are given. We will show that for some of these classes, the  frozen colourings are also Kempe frozen. Also in \cite{frozen}, an operation was given which transforms a $k$-colourable graph with a frozen $(k+1)$-colouring into a $(k+1)$-colourable graph with a frozen $(k+2)$-colouring, and preserves being $2K_2$-free. (There are some restrictions on the graph, the colouring, and the frozen colouring.)  We will show that the operation does the same for Kempe frozen colourings.

A set of vertices of a graph $G$ is called a {\em clique} if every pair of them is  adjacent. 
A {\em clique partition} of a graph is a partition of its vertices into cliques, and a partition into at most $k$ cliques is called a {\em k-clique partition}.  
A {\em frozen $k$-clique partition} of a graph $G$ is a partition of its vertices into $k$ nonempty cliques such that each vertex $v$ has a nonneighbour in every clique of the partition  other than the clique containing it; in other words, $V(G)\setminus N(v)$ contains a vertex of each of the $k$ cliques. A {\em Kempe frozen k-clique partition} of a graph $G$ is a partition of the vertices into $k$ nonempty cliques, such that the subgraph induced by any two of the cliques is a graph whose complement is connected.  

Because the classes of graphs of \cite{frozen} are dense, it is easier to visualize their complements, which are $C_4$-free graphs which admit frozen clique partitions.
%, where a clique partition is {\em frozen} if every vertex has a nonneighbour in each clique of the clique partition other than the clique of the partition it is in. 

\begin{remark} \label{rem:Kfrozen} 
{\em 
A Kempe frozen $k$-colouring of a graph corresponds to a Kempe frozen $k$-clique partition of its complement.
%A Kempe frozen colouring of a graph corresponds to a clique partition of its complement such that for each pair of cliques of the partition, the subgraph induced by the vertices in the two cliques is a graph whose complement is connected; we call such a clique partition a {\em Kempe frozen clique partition}. 
In the frozen clique partitions in this section, all cliques have size 2, so for the clique partitions to be Kempe frozen, for any two cliques of the clique partition, there must be at most one edge joining them, thus the subgraph induced by the vertices of the two cliques is either $2K_2$ or $P_4$.}
\end{remark}

%A \emph{Hamiltonian cycle} in a graph $G$ is a cycle which contains all the vertices of $G$.
\medskip

For an integer $q \ge 2$, $\overline{D_q}$  is the graph with $4q+2$ vertices $\{u_i: i=0,1,\ldots,q+1\} \cup \ \{\cup \{v_{i1}, v_{i2}, v_{i3}\}: i=1,2,\ldots,q\}$\\ whose edges are: 
\begin{itemize}[noitemsep,topsep=0pt]
\item  the edges of the Hamiltonian cycle $C$: $u_0$, $u_1$, \ldots, $u_{q+1}$, $v_{11}$, $v_{12}$, $v_{13}$, $v_{21}$, $v_{22}$, $v_{33}$,$\ldots$, $v_{q1}$, $v_{q2}$, $v_{q3}$, $u_0$.
\item edges $u_i v_{i2}$ for $i = 1, 2, \ldots, q$
\item edges $v_{i1}v_{i3}$ for $i = 1, 2, \ldots, q$
\end{itemize}

\vspace{5pt}
See Figure \ref{fig:D2} for $\overline{D_2}$ and and Figure \ref{fig:D3} for $\overline{D_3}$.
\vspace{5pt}
\begin{figure}[ht!]
\centering
\begin{tikzpicture}[scale=2.8]
\tikzstyle{vertex}=[circle, draw, fill=black, inner sep=0pt, minimum size=5pt]    
        \draw (0,0) circle (1); 
        \node[vertex, label=right:2](0) at (1,0) {};
    \node[vertex, label=above right:2](1) at (cos{36},sin{36}) {};
    \node[vertex, label=above:2](2) at (cos{72},sin{72}) {};
    \node[vertex, label=above:1](3) at (cos{108}, sin{108}) {};
    \node[vertex, label=above left:1](4) at (cos{144}, sin{144}) {};
    \node[vertex, label=left:1](5) at (cos{180}, sin{180}) {};
    \node[vertex, label=below left:4](6) at (cos{216}, sin{216}) {};
    \node[vertex, label=below:4](7) at (cos{252},sin{252}) {};
    \node[vertex, label=below:3](8) at (cos{288}, sin{288}) {};
    \node[vertex, label=below right:3](9) at (cos{324},sin{324}) {};    \draw(0)--(2);\draw(1)--(7);\draw(3)--(5);\draw(4)--(8);
\end{tikzpicture}
\hspace{0mm}
\begin{tikzpicture}[scale=2.8]
\tikzstyle{vertex}=[circle, draw, fill=black, inner sep=0pt, minimum size=5pt]    
        \draw (0,0) circle (1); 
        \node[vertex, label=right:4](0) at (1,0) {};
    \node[vertex, label=above right:2](1) at (cos{36},sin{36}) {};
    \node[vertex, label=above:3](2) at (cos{72},sin{72}) {};
    \node[vertex, label=above:3](3) at (cos{108}, sin{108}) {};
    \node[vertex, label=above left:1](4) at (cos{144}, sin{144}) {};
    \node[vertex, label=left:5](5) at (cos{180}, sin{180}) {};
    \node[vertex, label=below left:5](6) at (cos{216}, sin{216}) {};
    \node[vertex, label=below:2](7) at (cos{252},sin{252}) {};
    \node[vertex, label=below:1](8) at (cos{288}, sin{288}) {};
    \node[vertex, label=below right:4](9) at (cos{324},sin{324}) {};    
    \draw(0)--(2);\draw(1)--(7);\draw(3)--(5);\draw(4)--(8);
     
\end{tikzpicture}

\caption{A $C_4$-free graph $\overline{D_2}$ with a 4-clique-partition (left) and a Kempe frozen 5-clique-partition (right). The numbers indicate which clique a vertex is in. Equivalently, the numbers indicate a 4-colouring of the complement $D_2$ of the graph shown (left) and a Kempe frozen 5-colouring of $D_2$ (right).}
\label{fig:D2}

\begin{tikzpicture}[scale=2.8]
\tikzstyle{vertex}=[circle, draw, fill=black, inner sep=0pt, minimum size=5pt]    
        \draw (0,0) circle (1); 
        \node[vertex, label=right:3](0) at (1,0) {};
    \node[vertex, label=above right:3](1) at (cos{26},sin{26}) {};
    \node[vertex, label=above right:2](2) at (cos{52},sin{52}) {};
    \node[vertex, label=above:2](3) at (cos{78},sin{78}) {};
    \node[vertex, label=above:2](4) at (cos{104}, sin{104}) {};
    \node[vertex, label=above left:1](5) at (cos{130}, sin{130}) {};
    \node[vertex, label=above left:1](6) at (cos{156}, sin{156}) {};
    \node[vertex, label=left:1](7) at (cos{180}, sin{180}) {};
    \node[vertex, label=below left:6](8) at (cos{206}, sin{206}) {};
    \node[vertex, label=below left:5](9) at (cos{232},sin{232}) {};
    \node[vertex, label=below:5](10) at (cos{258}, sin{258}) {};
    \node[vertex, label=below:4] (11) at (cos{284},sin{284}) {};    
    \node[vertex, label=below right:4] (12) at (cos{310},sin{310}) {}; 
    \node[vertex, label=below right:3] (13) at (cos{336},sin{336}) {};  
    \draw(0)--(9);\draw(1)--(13);\draw(2)--(4);\draw(3)--(10);
    \draw(5)--(7);\draw(6)--(11);
     
\end{tikzpicture}
\hspace{0mm}
\begin{tikzpicture}[scale=2.8]
\tikzstyle{vertex}=[circle, draw, fill=black, inner sep=0pt, minimum size=5pt]    
        \draw (0,0) circle (1); 
        \node[vertex, label=right:3](0) at (1,0) {};
    \node[vertex, label=above right:5](1) at (cos{26},sin{26}) {};
    \node[vertex, label=above right:5](2) at (cos{52},sin{52}) {};
    \node[vertex, label=above:2](3) at (cos{78},sin{78}) {};
    \node[vertex, label=above:4](4) at (cos{104}, sin{104}) {};
    \node[vertex, label=above left:4](5) at (cos{130}, sin{130}) {};
    \node[vertex, label=above left:1](6) at (cos{156}, sin{156}) {};
    \node[vertex, label=below left:7](7) at (cos{180}, sin{180}) {};
    \node[vertex, label=below left:7](8) at (cos{206}, sin{206}) {};
    \node[vertex, label=below:3](9) at (cos{232},sin{232}) {};
    \node[vertex, label=below:2](10) at (cos{258}, sin{258}) {};
    \node[vertex, label=below:1] (11) at (cos{284},sin{284}) {};    
    \node[vertex, label=below right:6] (12) at (cos{310},sin{310}) {}; 
    \node[vertex, label=below right:6] (13) at (cos{336},sin{336}) {};  
    \draw(0)--(9);\draw(1)--(13);\draw(2)--(4);\draw(3)--(10);
    \draw(5)--(7);\draw(6)--(11);
     
\end{tikzpicture}

\caption{A $C_4$-free graph $\overline{D_3}$ with a 6-clique-partition (left) and a Kempe frozen 7-clique-partition (right). The numbers indicate which clique a vertex is in. Equivalently, the numbers indicate a 6-colouring of the complement $D_3$ of the graph shown (left) and a Kempe frozen 7-colouring of $D_3$ (right).}
\label{fig:D3}
\end{figure}

Consider the ($2q+1$)-colouring $\psi$ of $D_q$ which partitions the vertices into the following colour classes: \\
%\begin{itemize}
%\item 
$\{u_1, v_{12}\}, \{u_2, v_{22}\}, \ldots, \{u_q, v_{q2}\}, 
\{u_{q+1},v_{11}\}$, 
%\item 
$\{v_{13}, v_{21}\}, \{v_{23}, v_{31}\}, \dots, \{v_{q-1\mkern3mu3},
v_{q\mkern3mu1}\}, \{v_{q3},u_0\}$
%\end{itemize}

%defined as follows:
%\begin{itemize}[noitemsep,topsep=3pt]
%\item For $i=1, 2, \dots, q$, let $\psi(\{u_i, v_{i2}\}) = i$. 
%\item For $i=1, 2, \dots, q-1$, let $\psi(\{v_{i3}, v_{i+1\mkern3mu1}\}) = i+q$.
%\item Let $\psi(\{v_{q3},u_0\}) = 2q$ and let $\psi(\{u_{q+1},v_{11}\}) = 2q+1$.
%\end{itemize}

\begin{theorem}[\cite{frozen}]\label{thm:dq}
    For $q \ge 2$, $D_q$ is a 2$K_2$-free graph such that
\[ \chi(D_q) = \omega(D_q) =
\begin{cases}
    (3$q$+2)/2 & \textit{if } q~ \textit{is even} \\
    (3$q$+3)/2 & \textit{if } q~ \textit{is odd}
\end{cases}\]
and $\psi$ is a frozen colouring of $D_q$.
\end{theorem}

\begin{lemma} \label{lem:DqKempeFrozen}
For $q \ge 2$, the $(2q+1)$-colouring $\psi$ of $D_q$ is Kempe frozen.
\end{lemma}

\begin{proof}
Let $q \ge 2$. It is easily seen by considering $\overline{D_q}$ and Remark \ref{rem:Kfrozen} that for any two colours, $a$ and $b$, of the colouring $\psi$ of $D_q$, the subgraph induced by the vertices coloured $a$ or $b$ in $D_q$ is connected. 
\end{proof}

\begin{lemma}\label{lem:DqNotKempClass}
For $q \ge 2$, the set $\mathcal{C}_{2q+1}(D_q)$  
of all $(2q+1)$-colourings of $D_q$ 
does not form a Kempe class.
\end{lemma}

\begin{proof}
Let $q \ge 2$. Note that $\chi(D_q) < 2q+1$. By Lemma \ref{lem:DqKempeFrozen}, the colouring $\psi$ of $D_q$ is Kempe frozen and thus not Kempe equivalent to any minimum colouring of $D_q$.  So the set of all $(2q+1)$-colourings of $D_q$ does not form a Kempe class.  
\end{proof}

\begin{figure}[ht!]
\centering
\begin{tikzpicture}[scale=2.8]   
\tikzstyle{vertex}=[circle, draw, fill=black, inner sep=0pt, minimum size=5pt]    
        \draw (0,0) circle (1); 
        \node[vertex, label=right:3](0) at (1,0) {};
    \node[vertex, label=above right:3](1) at (cos{30},sin{30}) {};
    \node[vertex, label=above:2](2) at (cos{60},sin{60}) {};
    \node[vertex, label=above:2](3) at (cos{90}, sin{90}) {};
    \node[vertex, label=above left:2](4) at (cos{120}, sin{120}) {};
    \node[vertex, label=left:1](5) at (cos{150}, sin{150}) {};
    \node[vertex, label=below left:1](6) at (cos{180}, sin{180}) {};
    \node[vertex, label=below:1](7) at (cos{210},sin{210}) {};
    \node[vertex, label=below:4](8) at (cos{240}, sin{240}) {};
    \node[vertex, label=below right:4](9) at (cos{270},sin{270})
    {};
   \node[vertex, label=below right:4](10) at (cos{300},sin{300}) {}; 
   \node[vertex, label=below right:3](11) at (cos{330},sin{330}) {}; 
    \draw(0)--(6);\draw(1)--(11);\draw(2)--(4);\draw(3)--(9); \draw(5)--(7);\draw(8)--(10);
    
\end{tikzpicture}
\hspace{0mm}
\begin{tikzpicture}[scale=2.8]
\tikzstyle{vertex}=[circle, draw, fill=black, inner sep=0pt, minimum size=5pt]    
    \draw (0,0) circle (1); 
    \node[vertex, label=right:1](0) at (1,0) {};
    \node[vertex, label=above right:4](1) at (cos{30},sin{30}) {};
    \node[vertex, label=above:4](2) at (cos{60},sin{60}) {};
    \node[vertex, label=above:2](3) at (cos{90}, sin{90}) {};
    \node[vertex, label=above left:3](4) at (cos{120}, sin{120}) {};
    \node[vertex, label=left:3](5) at (cos{150}, sin{150}) {};
    \node[vertex, label=below left:1](6) at (cos{180}, sin{180}) {};
    \node[vertex, label=below:6](7) at (cos{210},sin{210}) {};
    \node[vertex, label=below:6](8) at (cos{240}, sin{240}) {};
    \node[vertex, label=below right:2](9) at (cos{270},sin{270}){};
    \node[vertex, label=below right:5](10) at (cos{300},sin{300}) {}; 
    \node[vertex, label=below right:5](11) at (cos{330},sin{330}) {}; 
    \draw(0)--(6);\draw(1)--(11);\draw(2)--(4);\draw(3)--(9); \draw(5)--(7);\draw(8)--(10);
\end{tikzpicture}

\caption{A $C_4$-free graph $\overline{Y_2}$ with a 4-clique-partition (left) and a Kempe frozen 6-clique-partition (right). Equivalently, a 4-colouring of the complement $Y_2$  (left) and a Kempe frozen 6-colouring (right).}
\label{fig:Y2}
% \end{figure}

% \begin{figure}[ht]
% \centering

\vspace{5mm}
\begin{tikzpicture}[scale=2.8]   
\tikzstyle{vertex}=[circle, draw, fill=black, inner sep=0pt, minimum size=5pt]    
        \draw (0,0) circle (1); 
         \node[vertex, label=right:4](0) at (1,0) {};
       \node[vertex, label=above:4](1) at (cos{20},sin{20}) {};
    \node[vertex, label=right:3](2) at (cos{40},sin{40}) {};
    \node[vertex, label=above:3](3) at (cos{60}, sin{60}) {};
    \node[vertex, label=above:3](4) at (cos{80}, sin{80}) {};
    \node[vertex, label=above:2](5) at (cos{100}, sin{100}) {};
    \node[vertex, label=above left:2](6) at (cos{120}, sin{120}) {};
    \node[vertex, label=above left:2](7) at (cos{140},sin{140}) {};
    \node[vertex, label=left:1](8) at (cos{160}, sin{160}) {};
    \node[vertex, label=left:1](9) at (cos{180},sin{180})
    {};
   \node[vertex, label=left:1](10) at (cos{200},sin{200}) {}; 
   \node[vertex, label=below left:6](11) at (cos{220},sin{220}) {}; 
   \node[vertex, label=below left:6](12) at (cos{240},sin{240}) {}; 
   \node[vertex, label=below:6](13) at (cos{260},sin{260}) {}; 
   \node[vertex, label=below:5](14) at (cos{280},sin{280}) {}; 
   \node[vertex, label=below:5](15) at (cos{300},sin{300}) {}; 
   \node[vertex, label=below right:5](16) at (cos{320},sin{320}) {}; 
   \node[vertex, label=below right:4](17) at (cos{340},sin{340}) {}; 
     \draw(0)--(9);\draw(1)--(17);\draw(2)--(4);\draw(3)--(12); \draw(5)--(7);\draw(6)--(15); \draw(8)--(10); \draw(11)--(13);
    \draw(14)--(16);
    
\end{tikzpicture}
\hspace{0mm}
\begin{tikzpicture}[scale=2.8]
\tikzstyle{vertex}=[circle, draw, fill=black, inner sep=0pt, minimum size=5pt]    
        \draw (0,0) circle (1); 
        \node[vertex, label=right:1](0) at (1,0) {};
    \node[vertex, label=right:6](1) at (cos{20},sin{20}) {};
    \node[vertex, label=above:6](2) at (cos{40},sin{40}) {};
    \node[vertex, label=above:3](3) at (cos{60}, sin{60}) {};
    \node[vertex, label=above:5](4) at (cos{80}, sin{80}) {};
    \node[vertex, label=above:5](5) at (cos{100}, sin{100}) {};
    \node[vertex, label=above left:2](6) at (cos{120}, sin{120}) {};
    \node[vertex, label=above left:4](7) at (cos{140},sin{140}) {};
    \node[vertex, label=left:4](8) at (cos{160}, sin{160}) {};
    \node[vertex, label=left:1](9) at (cos{180},sin{180})
    {};
   \node[vertex, label=left:9](10) at (cos{200},sin{200}) {}; 
   \node[vertex, label=below left:9](11) at (cos{220},sin{220}) {}; 
   \node[vertex, label=below left:3](12) at (cos{240},sin{240}) {}; 
   \node[vertex, label=below:8](13) at (cos{260},sin{260}) {}; 
   \node[vertex, label=below:8](14) at (cos{280},sin{280}) {}; 
   \node[vertex, label=below right:2](15) at (cos{300},sin{300}) {}; 
   \node[vertex, label=below right:7](16) at (cos{320},sin{320}) {}; 
   \node[vertex, label=below right:7](17) at (cos{340},sin{340}) {}; 
    \draw(0)--(9);\draw(1)--(17);\draw(2)--(4);\draw(3)--(12); \draw(5)--(7);\draw(6)--(15); \draw(8)--(10); \draw(11)--(13); \draw(14)--(16);
\end{tikzpicture}

\caption{A $C_4$-free graph $\overline{Y_3}$ with a 6-clique-partition (left) and a Kempe frozen 9-clique-partition (right). Equivalently, a 6-colouring of the complement $Y_3$ (left) and a Kempe frozen 9-colouring (right).}
\label{fig:Y3}
\end{figure}

We now define a second class of graphs. For $r \ge 1$, $\overline{Y_r}$ is the graph with $6r$ vertices \\ 
$\{\cup \{v_{i1}, v_{i2}, v_{i3}\}: i=1,2,\ldots,2r\}$\\ whose edges are: 
\begin{itemize}[noitemsep, topsep=0pt]
\item  the edges of a Hamiltonian cycle $C$: $v_{11}, v_{12}, v_{13}, v_{21}, v_{22}, v_{33},\ldots, v_{2r \mkern3mu 1}, v_{2r \mkern3mu 2}, v_{2r \mkern3mu 3}, v_{11}$
\item edges $v_{i1}v_{i3}$ for $i = 1, 2, \ldots, 2r$
\item edges $v_{i2}v_{i+r \mkern3mu 2}$ for $i = 1, 2, \ldots, r$
\end{itemize}

We refer to  $\{v_{i1}, v_{i2}, v_{i3}\}$ as \emph{triangle $i$}. Note that $\overline{Y_r}$ consists of a Hamiltonian cycle $C$ together with $2r$ edges which induce $2r$ vertex-disjoint triangles with consecutive pairs of edges of $C$, and $r$ more edges pairing the middle vertices $v_{i2}$ of ``opposite" triangles. 
%The number of edges of $\overline{Y_r}$  is $9r$. 

See Figure \ref{fig:Y2} for $\overline{Y_2}$  and Figure \ref{fig:Y3} for $\overline{Y_3}$. Note that $\overline{Y_1}$ is $\overline{C_6}$.\\

Consider the $3r$-colouring $\zeta$ of $Y_r$ which partitions the vertices into the following colour classes:
\\
$\{v_{12},v_{r+1 \mkern3mu 2} \}, 
\{v_{22},v_{r+2 \mkern3mu 2} \}, \ldots, 
\{v_{r2},v_{2r \mkern3mu 2} \},  
\{v_{13}, v_{21}\}, 
\{v_{23}, v_{31}\},\ldots,
\{v_{2r \mkern3mu 3}, v_{11}\}$.

\begin{theorem}[\cite{frozen}]\label{thm:Yr}
For $r \ge 2$, $Y_r$ is a $2K_2$-free graph with $\chi(Y_q)=\omega(Y_q)=2r$, and $\zeta$ is a frozen $3r$-colouring of $Y_r$.
\end{theorem}

Note that, $\overline{Y_1} \cong \overline{C_6}$ is the triangular prism, whose chromatic number is 3 and the set of 3-colourings $\mathcal{C}_{3}(\overline{Y_1})$ does not form a Kempe class (see Figure \ref{fig:prism}).

\begin{lemma} \label{lem:YrKempeFrozen}
For $r \ge 2$, the $3r$-colouring $\zeta$ of $Y_r$ is Kempe frozen.
\end{lemma}

The proof is similar to the Proof of Lemma \ref{lem:DqKempeFrozen}.

\begin{lemma}\label{lem:YrNotKempClass}
For $r \ge 2$, the set $\mathcal{C}_{3r}(Y_r)$ of all $3r$-colourings of $Y_r$ does not form a Kempe class.
\end{lemma}

\begin{proof}
Let $r \ge 2$. Similar to the proof of Lemma \ref{lem:DqNotKempClass}, note that $\chi(Y_r) < 3r$. By Lemma \ref{lem:YrKempeFrozen}, the colouring $\zeta$ of $Y_r$ is Kempe frozen and thus not Kempe equivalent to any minimum colouring of $Y_r$.  So the set of all $3r$-colourings of $Y_r$ does not form a Kempe class.  
\end{proof}

\begin{table} 
\footnotesize
    \centering
    \scalebox{0.88}{
    \begin{tabular}{|c|c|c|c|c|c|}
    \hline
        class & $q/r$ & $n$ & $\chi~~~$($=\omega$) & \# colours  & (\# colours in  \\  
        & &   & & in $K$-frozen & in $K$-frozen 
         \\
        & &  &     & colouring &  colouring) - $\chi$ 
         \\
          
         \hline
         & & & & &      \\
$D_q$ & $q \ge 2$ & $4q+2$ &  $(3q+2)/2$   & $n/2=2q+1$ &  $q/2$    \\
         & & & ~~~for even $q$;  &  & ~~~for even $q$;    \\
         & & & $(3q+3)/2$ &  & $(q-1)/2$    \\
         & & & ~~~for odd $q$  &  & ~~~for odd $q$    \\
& & & & &      \\
\hline
%%%%%
    & & & & &     \\
$Y_r$ & $r \ge 2 $ & $6r$ &  $2r$  & $3r$ &  $r$    \\
         & & & & &     \\
\hline    
    \end{tabular}}
    \caption{Parameters of $2K_2$-free graph classes}
    \label{tab:parameters}
\end{table}

We can now prove Theorem \ref{thm:2K2-free not KC}: 
    For all $p\ge 0$, there is a $k$-colourable 2$K_2$-free graph $G$ such that  $\mathcal{C}_{k+p}(G)$ does not form a Kempe class.

\begin{proof} 
For $p=0$, as noted above, $\overline{C_6}$ is a 3-chromatic graph and $\mathcal{C}_3(\overline{C_6})$ does not form a Kempe class.\\
For $p=1$, $D_2$ is a 4-chromatic graph and by Lemma \ref{lem:DqNotKempClass},  $\mathcal{C}_5({D_2})$ does not form a Kempe class.\\
For $p\ge 2$, $Y_p$ is a $2p$-chromatic graph and by Lemma \ref{lem:YrNotKempClass}, $\mathcal{C}_{3p}(Y_{p})$ does not form a Kempe class.
\end{proof}

In \cite{frozen}, the following operation was defined and Theorem \ref{thm:2K_2 increase} was proved.

\begin{operation}
\label{op:2K2}
Given a graph $G$ and nonadjacent vertices $x$ and $y$ in $G$, we define the following operation to create a new graph $G^\prime$. Define $G^\prime$ to be the graph $G$ together with two additional vertices  $u$ and $v$ and with edges $vx,~ xy$ and $yu$; join $u$ and $v$ to all vertices of $G-\{x,y\}$.
\end{operation}

Note that this operation corresponds to subdividing the edge $xy$ of $\overline{G}$, that is, replacing $xy$ by the path $x,u,v,y$.

Two sets $X$ and $Y$ of vertices of a graph $G$ are called {\em anticomplete} if no edge of $G$ has an end in each set and called {\em complete} if every vertex in one set is adjacent to every vertex in the other set. A universal vertex $w$ in a graph $G$ is a vertex which is adjacent to all other vertices of $G$.

\begin{theorem}[\cite{frozen}] \label{thm:2K_2 increase}
Let $G$ be a $k$-colourable graph with a $k$-colouring $\beta$ and a frozen $(k+1)$-colouring $\gamma$, and let $x$ and $y$ be nonadjacent vertices of $G$ such that $\beta(x) \neq \beta(y)$ and such that either 

\begin{itemize}
\item [(1)] $\gamma(x) \neq \gamma(y)$, or 
\item [(2)]	$\{x,y\}$ is a colour class of $\gamma$.
\end{itemize}

\noindent Then the graph $G^\prime$ of Operation \ref{op:2K2} is $(k+1)$-colourable and admits a frozen $(k+2)$-colouring.
\\Furthermore, 
\begin{itemize}
\item [(3)] if $G$ is $k$-chromatic, then $G^\prime$ is $(k+1)$-chromatic.
\item [(4)] if $G$ is $2K_2$-free and if in case (1), there is no edge $rs$ such that 
$\{r,s\}$ is anticomplete to $\{x,y\}$, then $G^\prime$ is $2K_2$-free. 
\end{itemize}
\end{theorem}

We now extend Theorem \ref{thm:2K_2 increase} to Kempe frozen colourings:

\begin{theorem} \label{thm:K2K_2 increase}
Let $G$ be a $k$-colourable graph with a $k$-colouring $\beta$ and a Kempe frozen $(k+1)$-colouring $\gamma$, and let $x$ and $y$ be nonadjacent vertices of $G$ such that $\beta(x) \neq \beta(y)$ and such that either 

\begin{itemize}
\item [(1)] $\gamma(x) \neq \gamma(y)$, or 
\item [(2)]	$\{x,y\}$ is a colour class of $\gamma$.
\end{itemize}

\noindent Then the graph $G^\prime$ of Operation \ref{op:2K2} is $(k+1)$-colourable and admits a Kempe frozen $(k+2)$-colouring. So the set $\mathcal{C}_{k+2}(G^\prime)$ of all $(k+2)$-colourings of $G^\prime$ does  not form a Kempe class.
\\Furthermore, 
\begin{itemize}
\item [(3)] if $G$ is $k$-chromatic, then $G^\prime$ is $(k+1)$-chromatic.
\item [(4)] if $G$ is $2K_2$-free and if in case (1), there is no edge $rs$ such that 
$\{r,s\}$ is anticomplete to $\{x,y\}$, then $G^\prime$ is $2K_2$-free. 
\end{itemize}
\end{theorem}

A $(k+1)$-colouring of $G^\prime$ is the colouring $\beta$ of $G$  extended by making $\{u,v\}$ a new colour class. A new Kempe  frozen $(k+2)$-colouring $\gamma^{\prime}$ of $G^\prime$ is obtained from the Kempe  frozen colouring $\gamma$ of $G$ by 
\begin{itemize}
\item In Case (1), making 
$\{u,v\}$ a new colour class.
\item In Case (2), give $x$ a new colour and assign this colour to $u$ as well; give $v$ the colour $\gamma(y)$.   
\end{itemize}

\begin{proof}
Assume the hypotheses. 
We only need to prove that $\gamma^{\prime}$ is a Kempe frozen $(k+2)$-colouring of $G^\prime$. \\ 
%Let $\mathcal{F}$ be the set of colour classes of colouring $\gamma$ of $G$ and let $\mathcal{F}^{\prime}$ be the set of colour classes of colouring $\gamma^{\prime}$ of $G^{\prime}$.\\
In Case (1): Since $\gamma$ is a Kempe frozen colouring of $G$, every pair of colour classes induces a connected subgraph, and this remains true when the edge $xy$ is added. 

We only need to check that any colour class of $\gamma$ together with  $\{u,v\}$ induces a connected graph. That is certainly true for any colour class of $\gamma$ other than the colour class containing $x$ and the colour class containing $y$ since both $u$ and $v$ are complete to $G \setminus \{x,y\}$. 
In any Kempe frozen colouring, if there is a colour class consisting of a single vertex, say $w$, then $w$ must be a universal vertex. 
In $G$, $x$ and $y$ are nonadjacent, so neither is a universal vertex, and thus there is vertex $x^\prime$ of $G$ different from $x$ in the colour class $I(x)$ of $\gamma$ containing $x$ and a vertex $y^{\prime}$ different from $y$ in the colour class $I(y)$ of $\gamma$ containing $y$. 
%Since $u$ is adjacent to every vertex of $G$ other than $x$, and in particular, is adjacent to $x^\prime \in I(x)$, and since $v$ is adjacent to every vertex of $I(x)$, it follows that $I(x) \cup \{u,v\}$ induces a connected subgraph of $G^{\prime}$. Similarly, $I(y) \cup \{u,v\}$ induces a connected subgraph of $G^{\prime}$. \\
In $G^{\prime}$, $u$ and $v$ are complete to $x^\prime$ and to $y^\prime$, and $x$ is adjacent to $v$, and $y$ is adjacent to $u$, so $I(x) \cup \{u,v\}$ induces a connected subgraph, and $I(y) \cup \{u,v\}$ induce a connected subgraph.
\\
In Case (2): Since $\gamma$ is a Kempe frozen colouring of $G$, any two colour classes of $\gamma$ induce a connected subgraph of $G$. Adding the edge $xy$ to $G$ does not change this, although $\{x,y\}$ is no longer a colour class. Thus any two colour classes of $\gamma^{\prime}$ induce a connected subgraph of $G^{\prime}$ except possibly if one of them is $\{x,u\}$ or $\{y,v\}$.   
The set $\{u,y,x,v\}$ induced a $P_4$ in $G^{\prime}$. 
Consider a colour class $I$ of $\gamma$ different from $\{x,y\}$. Since $\{x,y\} \cup I$ induces a connected subgraph of $G$, there must be a vertex $z$ of $I$ adjacent to both $x$ and $y$, and every other vertex of $I$ must be adjacent to at least one of $x$ and $y$. 
Now consider the union of the vertices of the colour classes $\{x,u\}$ and $I$ of $G^{\prime}$. Vertex $u$ is adjacent to every vertex of $I$ and vertex $z \in I$ is adjacent to $x$. Thus the subgraph induced by $\{x,u\} \cup I$ is connected in $G^{\prime}$. Analogously, the subgraph induced by $\{y,v\} \cup I$ is connected in $G^{\prime}$. Thus $\gamma^{\prime}$ is Kempe frozen.
\end{proof}

\section{Open problems}

%By Lemma \ref{lem:triangle-free}, there is a 3-colourable triangle-free graph for which the set of all 3-colourings does not form a Kempe class. We ask the following.

%\begin{question}
%    For $k\ge 4$, is there a $k$-colorable triangle-free graph for which the  set of all $k$-colourings does not form a Kempe class?
%\end{question}

By Theorem \ref{thm:2K2-free not KC}, there is no constant $c\ge 0$ such that for every 2$K_2$-free $k$-colourable graph $G$ the set of all ($k$+$c$)-colourings of $G$ form a Kempe class. For $k \ge 3$, graph $H_k$ of \cite{ReedConj} is a 3-chromatic triangle-free graph on $5k$ vertices and admits a Kempe frozen $k$-colouring. We ask the following.

\begin{question}
%For $H\in \{3K_1, triangle, C_4\}$, 
Where $H$ is either $C_4$ or $3K_1$, 
    is there a constant $c\ge 0$ such that for every $k$-colourable $H$-free graph $G$ the set of all ($k$+$c$)-colourings of $G$ forms a Kempe class?
\end{question}

\end{document}